\documentclass[11pt,a4paper]{article}

\usepackage[a4paper,top=3.8cm,bottom=3.8cm,left=3cm,right=3cm]{geometry}%
\usepackage{hyperref}

\usepackage[T1]{fontenc}
\usepackage{amsmath,amssymb,mathtools,amsthm,amsopn}
\usepackage{mathpazo}
\usepackage{graphicx}  
\usepackage[usenames,dvipsnames]{xcolor}

\usepackage{yfonts}
\usepackage{tocloft}
\usepackage{fancyhdr}
\usepackage{fourier-orns}

\usepackage{enumitem}

\usepackage[dayofweek]{datetime}
\usepackage{mathrsfs}
\renewcommand{\phi}{\varphi}
\DeclareMathOperator{\Int}{int}
\DeclareMathOperator{\length}{length}

\newcommand{\step}[1]{\par\medskip\noindent\it#1\rm}

\newcommand{\scalar}[2]{\langle#1,#2\rangle}

\newcommand{\abs}[1]{\lvert#1\rvert}

\newcommand{\z}{\zeta}
\renewcommand{\gg}{\mathfrak{g}}

\newcommand{\wh}{\widehat}

\newcommand{\e}{\varepsilon}

\newcommand{\s}{\sigma}

\newcommand{\la}{\lambda}

\newcommand{\wt}{\widetilde}

\newcommand{\ol}{\overline}

\renewcommand{\d}{\delta}
\newcommand{\Eucl}{\textup{Euc}}
 
\newcommand{\p}{\partial}

\newtheoremstyle{pippo}  
  {}       
  {}       
   {\sffamily}   
 {}        
  {\bfseries}  
  {.}   
  {1ex}       
  {}           
\newtheoremstyle{pluto}  {}{}
{\slshape}  {}{\bfseries}  {.} {1ex}    {}

\newtheorem{theorem}{Theorem}[section]
\newtheorem{proposition}[theorem]{Proposition}

\newtheorem{lemma}[theorem]{Lemma}

\newtheorem{corollary}[theorem]{Corollary}

\theoremstyle{pluto}
\newtheorem{definition}[theorem]{Definition}
\newtheorem{remark}[theorem]{Remark}
\newtheorem{example}[theorem]{Example}

\usepackage{mdwlist}

\newcommand{\R}{\mathbb{R}}
\newcommand{\G}{\mathbb{G}}

\newcommand{\N}{\mathbb{N}}

\renewcommand{\d}{\delta}
\renewcommand{\t}{\tau}

\renewcommand{\a}{\alpha}

\DeclareMathOperator{\Span}{span}
\DeclareMathOperator{\ad}{ad}

\numberwithin{equation}{section}

\let\oldbibliography\thebibliography
\renewcommand{\thebibliography}[1]{%
  \oldbibliography{#1}%
  \setlength{\itemsep}{0pt}%
}

\usepackage{titlesec}
\titleformat{\section}{%
\normalfont\large\bfseries}{\thesection.}{1em}{}
\titleformat{\subsection}{%
\normalfont\normalsize\bfseries}{\thesubsection.}{1em}{}

\begin{document}

\title{Multiexponential maps in Carnot groups \\ with  applications to convexity and differentiability  
 }
\author{Annamaria Montanari \thanks{Dipartimento di Matematica, Universit\`a di Bologna. Email \texttt{annamaria.montanari@unibo.it}} \and Daniele Morbidelli
\thanks{Dipartimento di Matematica, Universit\`a di Bologna. Email \texttt{daniele.morbidelli@unibo.it}}}

\date{}

\maketitle

\tableofcontents

\begin{abstract}
We analyze some properties of a class of \emph{multiexponential maps}  appearing naturally in the geometric analysis of Carnot groups.  
We will see that such maps can be useful in at least two interesting problems. First,  in relation to the analysis of some regularity properties of horizontally convex sets.  Then, we will show that our multiexponential maps can be used to prove the \emph{Pansu differentiability} of the subRiemannian distance from a fixed point.

\end{abstract}
%
%

\section{Introduction}

In this paper we discuss a class of \emph{multiexponential maps} in Carnot groups.  We introduce    two notions of  ``multiexponential regularity'', a stronger one and a weaker one,    and we show how the weaker one  ensures a ``cone property" for horizontally convex sets. Furthermore, we show that the stronger condition guarantees  the Pansu differentiability of the subRiemannian distance from the origin at the pertinent point.

Let $(\G, \cdot)=(\R^n, \cdot)$ be a Carnot group of step~$s$ and denote by $V_1$ the first (horizontal)
layer of its stratified Lie algebra $ \gg=V_1 \oplus \cdots \oplus V_s $.  See Section~\ref{sezione2} for the precise definition. Assume that $V_1$ is $m$-dimensional and denote by $X_1,\dots , X_m$ the left-invariant vector fields in $V_1$ such that $X_j(0)=\p_{x_j}$   for $j=1,2,\dots,m$. We define 
the $p$-th multiexponential map $\Gamma^{(p)}:(\R^m)^p\to \G=\R^n$ as 
\begin{equation*}
 \Gamma^{(p)}(u_1, u_2,\dots, u_p):= \exp(u_1\cdot X)\cdot \exp(u_2\cdot X)\cdot \cdots\cdot \exp(u_p\cdot X),
\end{equation*}
where given $u_j=(u_j^1, \dots, u_j^m)\in\R^m$, we denoted $u_j\cdot X=\sum_{k=1}^m u_j^k X_k\in V_1 $. Furthermore,  $\exp:\gg\to \G$ denotes the standard exponential map. See~\cite{BonfiglioliLanconelliUguzzoni}.
We are interested in those vectors $\xi\in \R^m$ such that     the following  holds:  there is $p\in\N$  such that the map $\Gamma^{(p)}$ is  a submersion at $(\xi, \xi, \dots, \xi)\in (\R^m)^p,$ namely 
\begin{equation}\label{subus} 
 d \Gamma^{(p)}(\xi, \xi, \dots, \xi):(\R^m)^p\to \R^n \quad\text{is onto}.
\end{equation}
We also consider a second, weaker condition: there is $p\in\N$ such that the map $\Gamma^{(p)}$ is locally open at $(\xi, \xi, \dots,\xi)\in(\R^m)^p$. Namely, 
\begin{equation}
 \label{saba} 
 \begin{aligned}
&\text{ for all $\e>0$ there is $\sigma_\e>0$ such  that }
 \\&   \Gamma^{(p)}(B_{\Eucl}(\xi, \xi, \dots, \xi),\e\big)  \supset 
  B_\Eucl (\Gamma^{(p)} (\xi, \xi, \dots, \xi),\s_\e ).                                                      \end{aligned}
\end{equation} 
  Here and hereafter we denote by $B_\Eucl$ Euclidean balls.   
In view of  the identification $\R^m\ni u\mapsto u\cdot X\in V_1$, sometimes we will refer to the submersion (respectively, local openness) conditions~\eqref{subus} and~\eqref{saba} by saying that the $p$-th multiexponential is a submersion (respectively, is locally open) at 
 $\xi\cdot X\in V_1$. Furthermore, by dilation properties in Carnot groups (see Section~\ref{sezione2}), property~\eqref{subus} (respectively,~\eqref{saba}) holds for   $p\in\N$ and $\xi\in\R^m\setminus\{0\}$ if and only if~\eqref{subus} (respectively,~\eqref{saba}) holds for $p\in\N$ and   $\xi/|\xi|\in\mathbb{S}^{m-1}$. 
  By elementary differential calculus,  if $\Gamma^{(p)}$ is a submersion at $(\xi, \dots, \xi)\in (\R^m)^p$, then 
$\Gamma^{(p)}$ is  locally open at $(\xi, \dots, \xi)\in (\R^m)^p$. In other words,~\eqref{subus} implies~\eqref{saba}. The opposite implication may fail, see next paragraph.

  In relation with the notions we have introduced above, it is also interesting to consider the path $\gamma_\xi(s)=\exp(s\xi\cdot X)$.  It is well known that such path is defined  for all $s\in\R$ and it is a global  length-minimizer for the Carnot-Carath\'eodory distance associated with the vector fields $X_1,\dots, X_m$. 
It is easy to realize that if the minimizer $\gamma_\xi$ is singular (i.e., abnormal) in the usual sense of the subRiemannian control theory (see \cite{AgrachevBarilariBoscain})    then
for all $p\in\N$
the multiexponential~$\Gamma^{(p)}$ is not   a submersion   at $(\xi, \dots, \xi)\in (\R^m)^p$.   See Remark~\ref{necesse}. 
On the other side, in a step-two Carnot group, if we take any $\xi\in \R^m$ such that the curve $\gamma_\xi$ is singular (abnormal), in~\cite[Theorem~2.1 and Remark~2.2]{Morbidelli18} it is shown that there is $p\in\N$ such that $\Gamma^{(p)}$ is locally open at $(\xi, \dots, \xi)\in (\R^m)^p$. This provides  examples where the local openness~\eqref{saba} holds for some $p\in \N$, while the stronger submersion condition~\eqref{subus} fails for all $p\in\N$.

Here is the statement of our  first result, on  the cone property  for horizontally convex sets. 

 \begin{theorem}\label{teoremetto} 
Let $\G=(\R^n,\cdot )$ be a Carnot group and assume that   for some $V=\xi\cdot X\in V_1$, there is $p\in\N$ such that   the local openness condition~\eqref{saba} holds.    
Let also~$A\subset\G$ be a horizontally convex set such that for some  $  \ol x\in \ol A$ we have $\ol x\cdot \exp(V)\in\Int(A)$. Then   there  is $\e>0$ 
such that    for all $x\in   \ol A$ with $d(x,\ol x)<\e $ we have   
\begin{equation}\label{ggo}  
  \begin{aligned}
&\bigcup_{0<s<1}B\Big(x \cdot\exp(sV ), \e s\Big)\subset \Int(A).
  \end{aligned}
\end{equation}\end{theorem}
   In the statement of the theorem $B(x, r)$ denotes the ball  centered at~$x$ 
and with radius~$r$  with respect to the subRiemannian distance~$d$ defined by the vector 
fields $X_1,\dots, X_m$.  
The set  appearing in the left-hand side of~\eqref{ggo} is a (truncated)
subRiemannian twisted cone  and the \emph{horizontal segment} $\{  x\cdot\exp(sV) :0\leq s\leq 1\}$ can be 
considered as the  axis  of the 
cone.    Note that a horizontal  segment    does not need to be     an Euclidean 
segment, as explicit examples will show later.    Finally, note that $\Int (A)$ and $\ol A$ denote the interior and the closure of a set~$A$ in the Euclidean topology (which is the same of that induced by the subRiemannian distance).

The cone property appears in several interesting  questions in the geometric analysis of subRiemannian spaces:
\begin{itemize}[nosep]
\item[(i)] in the theory of sets with finite horizontal
perimeter in Carnot groups  (see 
\cite{MontiVittone12};
\item[(ii)] in the intrinsic version of Rademacher's 
theorem in 
the case of the Heisenberg group (see \cite{FranchiSerapioniSerracassano11});

\item[(iii)] in the definition of intrinsic Lipschitz  
continuous graphs inside Carnot groups (see \cite{FranchiSerapioni16} and  the 
references therein).
\end{itemize}

Let us observe that the cone property~\eqref{ggo} is trivial for convex sets in the Euclidean space.
In such case the submersion propery \eqref{subus} is always fulfilled with $p=1$.  On the other side,  in subRiemannian settings we will see that if we do not assume the local openness~\eqref{saba} for some $p\in \N$, then Theorem \ref{teoremetto} can be false. Counterexamples will be presented  
 in Sections~\ref{ggiorgio} and~\ref{gggiorgio}.

The proof of Theorem \ref{teoremetto} is based on an argument used by Cheeger and Kleiner in \cite{CheegerKleiner10}, in the context  of classification of monotone sets in the Heisenberg group. Namely, the mentioned authors used the maps $\Gamma^{(p)}$ with $p=2$ to prove a qualitative version of Theorem~\ref{teoremetto}  in the three-dimensional Heisenberg group. Then, a similar argument with maps $\Gamma^{(p)} $ with $p\geq 2$
has been used by the second author to prove a cone property in general two-step Carnot groups in  \cite{Morbidelli18}. Here we adapt the argument in order to   show   a statement which holds in any Carnot group of any step.

 In this paper we are able to find a new interesting class of models,   known as \emph{filiform Carnot groups of first type},    where  the hypotheses of such theorem are fulfilled. 
In order to state our result,   let us introduce some notation.    Consider in $\R^{p+2}$, equipped with coordinates $(x,y, t_1, t_2, \dots, t_p)$, the vector fields
\begin{equation}\label{osteo}  
\begin{aligned}
&
  X =\p_x\quad \text{and}\quad Y=\p_y+ x \p_{t_1} + \frac{x ^2}{2}\p_{t_2}+\cdots + 
  \frac{x^p}{p! }\p_{t_p}=\p_y+\sum_{ k=1}^p\frac{x^k}{k!}\p_{t_k}.
\end{aligned}
 \end{equation}
Given the vector fields $X,Y$, there is a   Carnot group $(\R^{p+2},\cdot)$ of step $p+1$  such that  $V_1=\Span\{X,Y\}$. See the discussion in Section~\ref{zibibbo}, for details. Note that if $p=1$ then we get the Heisenberg group. If $p=2$, then we get a Carnot group of step three which is known as the Engel group. Otherwise, we will call it the \emph{filiform  group} of step $p+1$.
  The  nilpotent stratified Lie algebra generated by $X$ and $Y$ is known as \emph{filiform algebra of first type}. See Vergne's paper~\cite{Vergne} and see also \cite[Section~6]{DLDMV19}.

\begin{theorem}\label{filetto} 
 Let $p\geq 2$,   let $X$ and $Y$ be the vector fields in \eqref{osteo}. 
Consider  a horizontal left invariant vector field $Z= uX+vY$. Then,   there is $q\in\N$ such that~\eqref{subus} holds 
if and only if $u\neq 0$.
\end{theorem}
  The proof  that $u\neq 0$ is a sufficient condition in Theorem \ref{filetto} is proved in Theorem \ref{cassetto}, while  the fact that it is also  a necessary condition is proved in Section~\ref{ggiorgio}. 

Then we get the following corollary.
\begin{corollary}\label{unoo} 
 Let $A\subset \R^{p+2} $ be a \emph{horizontally convex} set with respect to the pair of vector fields 
 in~\eqref{osteo}. Assume that $(z,t)=(x,y,t_1,\dots, t_p)\in \p A$ and assume that 
  there is $Z:= u X+vY$ such that 
 $ (z,t)\cdot\exp( Z)\in\Int(A)$ and  $u\neq 0$.
 Then there is $\e>0$ such that 
 \begin{equation*}
  \begin{aligned}
&\bigcup_{0<s\leq 1 }B\Big((z,t) \cdot\exp(s Z), \e s\Big)\subset \Int(A).
  \end{aligned}
\end{equation*}
\end{corollary}

Corollary \ref{unoo}  generalizes the result  proved by Arena, Caruso and Monti in \cite{ArenaCarusoMonti12} and by the second author in \cite{Morbidelli18}.





Next we pass to a description of our second set of results.  In Section~\ref{dudu}, we   will prove   the following statement.     

 \begin{theorem}\label{oppiaceo} 
 Let $ (\R^n, \cdot)$ be a Carnot group and assume that   for some $V=\xi\cdot X\in V_1$ condition~\eqref{subus} holds for some $p\in\N$.  
 Then the subRiemannian distance from the origin is Pansu differentiable at $\exp(V)$.
\end{theorem}

In particular we shall apply our statement to get a new proof  of some recent results by   Pinamonti and Speight in   \cite{PinamontiSpeight18}.
\begin{corollary}[\cite{PinamontiSpeight18}]
 \label{prof} Let $ (\R^n, \cdot)$ be a filiform Carnot group. If $u\neq 0$, then  the subRiemannian distance   from the origin is Pansu differentiable at $\exp(uX+vY)$. 
\end{corollary}
The proof of Corollary~\ref{prof}   follows immediately putting together Theorems~\ref{filetto} and~\ref{oppiaceo}. Our argument seems to be somewhat simpler than the original one in \cite{PinamontiSpeight18}.

In the setting of Carnot groups of step two, Le Donne, Pinamonti and Speight \cite{LeDonnePinamontiSpeight17}   proved   
that the subRiemannian distance is differentiable at $\exp(V)$ for any   $V\in V_1$.    Such statement can not be obtained as a consequence  of Theorem~\ref{oppiaceo}, because it may happen that the curve $\gamma(s)=\exp(sV)$ is abnormal and  
in such case there is no $p\in \N$ such that the $p$-th multiexponential satisfies the submersion condition~\eqref{subus} at the corresponding point $(\xi,\dots, \xi)\in(\R^m)^p$, where $\xi\cdot X= V$.  
However, the Pansu differentiability  can be proved  at abnormal points using the local openness~\eqref{saba}  of the maps $\Gamma^{(p)}$.   The argument, which
can have some independent interest in other questions related with two-step Carnot groups, will be carried out in
  Section~\ref{duofold}. Here is the statement.
\begin{theorem}[\cite{LeDonnePinamontiSpeight17}] \label{pasodos}      If $(\R^n, \cdot)$ is a Carnot group of step two,  then the subRiemannian distance from the origin
  is Pansu differentiable at the point $\exp(V)$ for any nonzero  $V\in V_1$.
\end{theorem}

                          

\section{Preliminaries}\label{sezione2}

\paragraph*{Control distances.}
 Let $X_1,\dots, X_m$ be a family of   smooth   vector fields in $\R^n$.  Assume that  the vector fields are linearly independent at every point. A Lipschitz path $\gamma:[a,b]\to\R^n$  is said to be \emph{horizontal} if it satisfies almost everywhere in $[a,b]$ the ODE $\dot\gamma =\sum_{j=1}^m u_j(t) X_j(\gamma)$, where the \emph{control}  $u=(u_1, \dots, u_m)$ belongs to $  L^1((a,b),\R^m)$. 
 In such case,   define the subRiemannian length of $\gamma$ as $
  \length(\gamma):=\int_a^b|u(s)|ds
$ and 
given two points~$x$ and~$y\in\R^n$
 the \emph{subRiemannian distance}   
$  d(x,y)=\inf \{\length(\gamma)\}$, where the infimum is taken on all horizotal curves connecting $x$ and $y$.

\paragraph*{Carnot groups.} 
%
Let us recall the definition of  Carnot  group of step~$s\geq 2$. See~\cite[Section~1.4]{BonfiglioliLanconelliUguzzoni} for more details. Let $\G:=(\R^n,\cdot)$ be a Lie group with identity $0\in\R^n$. Assume that   $\R^n$ can be written as  $\R^n= \R^{m_1}\times\R^{m_2}\times\cdots\times\R^{m_s}\ni(x^{(1)},\dots, x^{(s)})$  and require that for all $\la>0$ the dilation map~$\delta_\la$ defined as
\begin{equation*}
x= (x^{(1)},x^{(2)},\dots, x^{(s)} )\mapsto
\d_\la(x):=(\lambda x^{(1)},\la^2 x^{(2)},\dots, \la^s x^{(s)} )
\end{equation*}
 is a group automorphism
 of $\G$ for all~$\la>0$. Let $m=m_1$ and let $X_1, X_2,\dots, X_m$ be the left-invariant vector fields such that $X_j= \p_{x_j}$ at the origin   for $j=1,\dots, m$.   
 We assume that the family $X_1, \dots, X_m$ satisfies the H\"ormander condition. 
 It is well known that the Lie algebra $\gg$ of $\G$ has a natural 
 stratification $\gg=V_1\oplus\cdots\oplus V_s$, where $V_1=\Span\{X_1,\dots, X_m\}$
  and $[V_1, V_j]=V_{j+1}$ for all $j\leq s-1$. Here $V_k$ denotes the span of the left invariant commutators of length~$k$. 
 
\paragraph{Carnot groups of step 2. }  
   Let us consider $(x^{(1)}, x^{(2)})= (x,t)\in \R^m_x\times\R^\ell_t$. Assume that we are given a map $Q=(Q^1,\dots, Q^\ell):\R^m\times\R^m\to \R^\ell$, bilinear and skew-symmetric. Assume also that 
   \begin{equation}
\label{horma}
\Span \{ Q(e_j, e_k):1\leq j<k\leq m\}=\R^\ell.
\end{equation} 
We can define the law
\begin{equation}\label{steppa} 
 (x,t)\cdot(\xi,\tau):=(x+\xi, t+\tau+Q(x,\xi)).
\end{equation}
The vector fields $\displaystyle  X_k= \p_{x_k}+\sum_{j=1,\dots, m,\; \a=1,\dots,\ell}  
Q^\a(e_j,e_k) x_j\p_{t_\a}$,   as $k=1,2,\dots, m$,    are  a basis of $V_1$   satisfying $X_k(0)=\p_{x_k}$.   
Another standard computation shows that the condition~\eqref{horma} ensures that the H\"ormander condition holds.   Namely  $\Span\{X_i, [X_j, X_k]:i,j,k=1,\dots, m \}=\R^n$ at any point.  

 The easiest example of two-step Carnot group is the \emph{Heisenberg group}, where $\R^m\times\R^\ell = \R^2\times \R$ and $Q((x_1, x_2), (\xi_1, \xi_2))= x_1\xi_2 - x_2\xi_1$.

 \paragraph*{Pansu differentiability.} It has been shown by Pansu~\cite{Pansu89}  that, given a   Carnot group $\G=(\R^n,\cdot)$ with dilations $\delta_\la$ and given a map $f : \G\to \R$   that   is Lipschitz-continuous with respect to the subRiemannian distance, then the map~$f$ is Carnot differentiable $\mathcal{L}^n$-almost everywhere. Namely, for almost all  $x\in\R^n$ there exists a~$\G$-linear map $T
 :\G\to\R$ such that 
 \begin{equation*}
\lim_{y\to 0}\frac{  f(x\cdot y)- f(x)- Ty}{d(0,y)}=0.
 \end{equation*}
Recall that a map $T:\G\to\R$ is said to be 
 $\G$-linear if it
satisfies $T(x\cdot y)= T(x) + T(y)$ and $T(\delta_\la x)= \la T(x)$ for all~$x,y\in\G$ and~$\la>0$.
   Since by elementary properties of metric spaces,  the distance function from a fixed set (or from a point)  is Lipschitz-continuous, Pansu's theorem of differentiability ensures that the distance function is Pansu differentiable almost everywhere.  
  
 \paragraph{Horizontal lines and  horizontally convex sets.} Let $X_1, \dots, X_m$ be the horizontal left-invariant vector fields on a Carnot group. Given  $u=(u_1, \dots, u_m)\in \R^m$ we denote by $u\cdot X:=\sum_{j=1}^m u_jX_j$. Note that any horizontal left-invariant vector field can be written in the form $u\cdot X$ for some $u\in\R^m$. A \emph{horizontal line} (briefly, a \emph{line}) is any set of the form 
   $  \ell:=\{  x\cdot\exp(s u\cdot X) :s\in\R \}$ for some $x\in\G$. 
    Observe that
  not all Euclidean lines are  horizontal lines.  On the 
  other side,  
  in Carnot groups of step at least three, it can happen that a line is not an 
  Euclidean line. 
  
  We say that the points $x$ and $y\in\G$ are \emph{horizontally aligned}  
if they belong to the same horizontal line $\ell  = \{\gamma(s):=\ol x\cdot(\exp(sV)):s\in\R\}$ for some $\ol x\in\R^n$ and $V\in V_1$. 
 A set $A\subset \G$ is \emph{horizontally convex} if for all horizontally aligned 
  points $x=\gamma(a) $ and~$ y=\gamma(b)\in\ell$, then   the horizontal 
  segment~$\gamma([a,b])$ connecting~$x$ and~$y$ is contained in $A$.
 
 \paragraph*{Multiexponentials.} 
 Given a Carnot group $\G$ with   a basis of left-invariant horizontal   vector fields~$X_1,\dots, X_m$, and a fixed   number $p\in\N$, we define for   all vectors  $u=(u_1, u_2, \dots, u_p)\in(\R^m)^p$, the map
 \begin{equation*}
  \Gamma^{(p)} (u_1, \dots, u_p):= \exp(u_1\cdot X)\cdot \exp(u_2\cdot X)\cdots\exp(u_p\cdot X) =e^{u_p\cdot X}\cdots e^{u_1\cdot X}(0),
 \end{equation*}
where $e^Z x$ denotes the value at time $t=1$ of the integral curve of $Z$ starting from $x$ at $t=0$, while $\exp:\gg\to \G$ denotes the exponential map of the Lie group theory. See~\cite{BonfiglioliLanconelliUguzzoni}. The map $\Gamma^{(p)}$ can be thought of as defined on the product $(V_1)^p$. 
  Finally, observe the dilation formula  
\begin{equation}
  \label{caracalla}\Gamma^{(p)}
  (\la u_1, \dots, \la u_p)=\delta_\la\big( \Gamma^{(p)}(u_1,\dots, u_p)\big),  
 \end{equation}
 for all $p\in\N$, $u_1, \dots, u_p\in\R^m$ and $\la>0$.   Formula 
 \eqref{caracalla} with $p=1$ follows from \cite[Lemma~1.3.27]{BonfiglioliLanconelliUguzzoni} and the general case is a consequence of the fact that $\delta_\la:\G\to\G$ is a morphism of $(\G,\cdot)$).

\begin{definition}  
Given a Carnot group $(\G, \cdot, \delta_\la)$, $\xi\in\R^m\setminus\{0\}$ and $p\in\N$, we say that  the
  $p$-th multiexponential   is   a submersion at $(\xi, \dots, \xi)$ if 
\begin{equation}\label{sommergo} 
 d\Gamma^{(p)}(\xi, \xi,\dots, \xi)
 :(\R^m)^p \to\G
 \end{equation}  
 is onto. We say that the $p$-th   multiexponential   is locally open at $(\xi,\dots, \xi)\in (\R^m)^p  $ if for all $\e>0$ there is $\s_\e>0$ such that 
 \begin{equation*}
  \Gamma^{(p)}\big(B_{\Eucl}(\xi,\xi, \dots, \xi),\e \big)\supset
  B_\Eucl\big( \Gamma^{(p)}(\xi,\xi, \dots, \xi),\s_\e\big) .
 \end{equation*}
 \end{definition}
 Well known   properties of  dilations   (see~\eqref{caracalla}) show that
$\Gamma^{(p)}$ is a submersion (respectively, locally open) at $(\xi, \dots, \xi)\in(\R^m)^p$ if and only if 
$\Gamma^{(p)}$ is a submersion (respectively, locally open) at $(\lambda\xi, \dots, \lambda \xi)\in(\R^m)^p$ for any $\lambda > 0$.

\begin{remark}
 \label{necesse} If condition \eqref{sommergo} holds for some $p\in\N$ and $\xi\in\R^m$, then the curve $\gamma(s)=\exp(s\xi\cdot X)$ is nonsingular in the sense of SubRiemannian control theory. To check this remark, it suffices to consider the end-point map $E: L^2
 ((0,1), \R^n)$ defined as follows. 
 Given $u\in L^2$ we put  $E(u) = \gamma(1)$, where $\gamma:[0,1]\to \R^n $   solves the problem $\dot\gamma(t)$$=\sum_{j=1}^m $$u_j(t) X_j(\gamma(t))$, a.e.~ and $\gamma(0)=0$. 
 Let us restrict the end point map~$E$ to the finite dimensional affine subspace 
 \begin{equation*}
\Sigma=\Bigl\{u\in L^2((0,1), \R^m) : u(t)=\xi+\sum_{j=1}^p
h_j\mathbb{1}_{[\frac{j-1}{p}, \frac jp]}(t),
\text{ with }(h_1, \dots, h_p)\in (\R^m)^p\Bigr\}.
 \end{equation*}
Let $\ol u\in L^2((0,1), \R^m)$, $\ol u(t)= \xi$  for $t\in(0,1)$.
Looking at the restriction of $E$ to $\Sigma$, {it} turns out easily that the image of $L^2$ through the differential $dE(\ol u) $ of the end point 
map contains   the image of $(\R^m)^p$  through $d\Gamma^{(p)}(\xi, \dots, \xi) $.
  Thus  \eqref{sommergo}  implies that $dE(\ol u):L^2\to \R^n$ is onto.
\end{remark}

 \paragraph*{M\'etivier groups.}
 A  two-step Carnot group, see~\eqref{steppa}, is said to be of M\'etivier type if for all $t\in\R^\ell$ and for all $x\neq 0$ there is a solution $y\in\R^m$ of the system $Q(x,y)= t$.
M\'etivier groups were introduced in \cite{Metivier80}. 
The most elementary example of M\'etivier group is  the Heisenberg group, while the easiest example of non-M\'etivier group is $\R^3_x\times\R$, with the map 
$Q((x_1, x_2, x_3), (y_1, y_2, y_3))= x_1y_2 -x_2 y_1$. Here, taking $x= (0,0,1)$, we see that $Q(x,y)=0$ for all $y$.

  Note that   in a two-step  group of M\'etivier type the map $\Gamma^{(2)}$ is   
  a submersion at any $
  (\xi,  \xi)\in(\R^m)^2$ with $\xi\neq 0$.    Indeed, differentiating the quadratic map $\Gamma^{(2)}(u,v)= (u+v, Q(u,v))$, we have
 \begin{equation*}
(u,v)\mapsto  d\Gamma^{(2)}(\xi,\xi)(u,v)= (u+v, Q(\xi, v)+ Q(u,\xi))= (u+v, Q(u-v, \xi)), 
 \end{equation*}
and   for all $\xi\neq 0$ this map is onto because the function $y\mapsto Q(y, \xi)$
 is onto.   Vice-versa, if in a two-step Carnot group  the M\'etivier condition fails, i.e. there is $\ol u\in \R^m$ such that  the map $  Q(\ol u,\cdot ):\R^m\to \R^\ell$ 
 is not onto, then, considering the constant control $u(t)=\ol u$ for $t\in[0,1]$, 
 by \cite[Proposition~3.4]{MontanariMorbidelli15},  we have $\operatorname{Im}dE(u)=\R^m\times\{  Q(\ol u, y):y\in\R^m\}$, which is a strict subset of  $\R^m\times\R^\ell$.   Then for all $p\in \N$, $\Gamma^{(p)}(\ol u, \dots, \ol u)$ is not a submersion, i.e.,~\eqref{subus} fails. 
  
 \section{ Multiexponentials in filiform groups. }\label{zibibbo} 
 In this section we introduce filiform Carnot groups and we 
  discuss  multiexponentials in that setting.
  As observed before the statement of Theorem~\ref{filetto}, we work on filiform groups of the first type.

Let us consider in $\R^{p+2}$ equipped with coordinates $(x,y, t_1, t_2, \dots, t_p)$ the vector fields
\begin{equation}\label{campetti} 
\begin{aligned}
&
  X =\p_x\quad \text{and}\quad Y=\p_y+ x \p_{t_1} + \frac{x ^2}{2}\p_{t_2}+\cdots + \frac{x^p}{p! }\p_{t_p}=\p_y+\sum_{ k=1}^p\frac{x^k}{k!}\p_{t_k}.
\end{aligned}
 \end{equation}
where $(z,t)= (x,y, t_1, \dots, t_p)\in\R^{p+2}$. 
 Let us denote $\ad_X Y:= [X,Y]$ and $\ad_X^k Y:= [X, \ad_X^{k-1}Y]$ for $k\geq2$.
 A computation shows that for $j=1,\dots, p$, we have  
\[
\begin{aligned}
\ad^j_{X }Y =\p_{t_j } +\sum_{k=j+1}^p \frac{x^{k-j}}{(k-j)!}\p_{t_k }.
\end{aligned}
 \]
In particular $ \ad^p_{X }Y = \p_{t_{p+1}}$. 
The vector fields $X$ and $Y$ generate a nilpotent  filiform Lie algebra of  step $p+1$.  
This is the structure denoted with $\mu_0$ in \cite[Corollaire  1, p.~93]{Vergne}.  
Defining in $\R^{p+2}=\R^2_{x,y}\times\R^p_t$ the binary law
\begin{equation}\label{groppo} 
\begin{aligned}
 (x,y,t)\cdot(\xi,\eta,\tau)
   = &\Big(x+\xi, y+\eta, 
t_1+\t_1+x\eta,t_2+\t_2+\frac{x^2}{2}\eta+ x\t_1,
\\& \qquad \qquad\dots, t_k+\t_k +\frac{x^k}{k!}\eta+ \sum_{j=1}^{k-1} \frac{x^{k-j}}{(k-j)!}\t_j,\dots \Big),
\end{aligned}
\end{equation} 
where $ k=2,\dots, p$, it turns out that $(\R^{p+2}, \cdot )$ is a Carnot group of step~$p+1$.
 
Particular familiar instances of filiform groups occur when $p=1$, and then we have the law
\begin{equation*}
\begin{aligned}
& (x,y,t)\cdot(\xi,\eta,\tau)= (x+\xi, y+\eta, t+\tau + x\eta ),\end{aligned}
\end{equation*}
 with horizontal  vector fields $ 
X=\p_x$  and $Y=\p_y+x\p_{t_1}$,
which after a linear change of variables becomes the familiar Heisenberg group. A second particular case is the so-called Engel group, which has step $p+1=3$ and whose group law is
\begin{equation*}
\begin{aligned}
 &(x,y,t_1, t_2)\cdot (\xi, \eta, \t_1, \t_2)= \Big(x+\xi, y+\eta, t_1+\t_1 + x\eta 
 ,t_2+\t_2+  x \t_1+\frac{x^2}{2}\eta \Big ),
\end{aligned}
 \end{equation*}
  with  horizontal vector fields $X=\p_x$ and $  Y=\p_y +x\p_{t_1}+\frac{x^2}{2}\p_{t_2}.$

The associative property of the law~\eqref{groppo} can be checked easily if we identify
\begin{equation}\label{amma} 
(x,y, t_1, t_2, \dots, t_p)\sim  \begin{bmatrix}
1 &  0 & 0& 0&\cdots &0
\\
y & 1& 0& 0& \cdots &0
\\
t_1 & x& 1& 0&\cdots &0
\\ 
t_2 & x^2/2  & x& 1& \cdots& 0
\\ t_3 &  x^3/ 3! & x^2/2 & x& \cdots &0
\\ \vdots &\vdots &\vdots &\vdots&\ddots & 0
\\ t_p & \frac{x^p}{p!} & \frac{x^{p-1}}{(p-1)!} &  \frac{x^{p-2}}{(p-2)!} & \cdots &1
 \end{bmatrix}\in\R^{(p+2)\times (p+2)}.
\end{equation} 
Under \eqref{amma}, the binary law~\eqref{groppo} can  be identified with  the matrix product. See \cite[Section~4.3.5 and~4.3.6]{BonfiglioliLanconelliUguzzoni}.

Define for 
$(w_1 , w_2 , \dots, w_q)\in  \R^{2q}$
\begin{equation}\label{pezzi2} 
 \Gamma^{(q)} (w_1, \dots, w_q) = e^{w_q\cdot Z }\cdots e^{w_1 \cdot Z}(0)
 =\exp(w_1\cdot Z)\cdot \exp(w_2\cdot Z)\cdots \cdot\exp(w_q\cdot Z)
\end{equation}
where  $w_k= (u_k,v_k)\in\R^2 $ and  $w_k\cdot Z=u_k  X+v_k Y$.

\begin{theorem} \label{cassetto}  
 Fix $p\in\N$ 
 and consider the vector fields in~\eqref{campetti}. Let $\zeta=(\xi ,\eta)\in\R^2$ such that $\xi \neq 0$. Then the map $\Gamma^{(p+1)}:\R^{2p+2}\to \R^{p+2} $ defined
 in~\eqref{pezzi2} 
 is a submersion at $(\zeta,\zeta,\dots, \zeta)\in   \R^{2p+2}  $. 
\end{theorem}

\begin{remark}
Note that   if $p\geq 2$ there is no $q\in\N$ such that the
map $\Gamma^{(q)}$  is a submersion at $((0,\eta), (0,\eta), \dots, (0,\eta))\in(\R^2)^q$ for some $\eta\neq 0$. Indeed, if this would  happen, 
  by Remark~\ref{necesse}  
we would contradict the well-known fact that the curve    $\gamma(s):= \exp(sY) $    is a singular extremal for the subRiemannian length minimization problem. See~\cite{LiuSussmann}.       
\end{remark}

\begin{proof}[Proof of Theorem \ref{cassetto}.]
We have to show that the linear map 
$d\Gamma^{(p+1)}(\zeta, \dots, \zeta):\R^{2(p+1)}\to \G$ is onto. We claim that the square matrix
 \begin{equation}\label{jacc} 
M(\zeta):=\left [\frac{\p \Gamma^{(p+1)}}{\p u_1},\frac{\p \Gamma^{(p+1)}}{\p v_1},\frac{\p \Gamma^{(p+1)}}{\p v_2},  
 \frac{\p \Gamma^{(p+1)}}{\p v_3 },\dots, \frac{\p \Gamma^{(p+1)}}{\p v_{p+1}}
 \right](\zeta, \zeta, \dots, \zeta)  \in\R^{(p+2)\times(p+2)}                                                                
 \end{equation}   is { non singular.} Since the matrix above is formed taking $p+2$ of the $2(p+1)$ columns of the Jacobian matrix, 
 the statement will follow immediately. 
 
 A first calculation shows that for $w=(u,v)\in\R^2$ and $(z,t)  =(x,y,t) 
 \in \R^2\times \R^p$ 
we have 
 \begin{equation*}
\begin{aligned}
  e^{w\cdot Z} (z,t) &= \bigg   ( x +u, y+v,  t_1 + v\int_0^1(x +s u)ds,
  t_2+ v \int_0^1\frac{(x +s u)^2}{2!} ds, \dots, 
\\& \qquad\qquad \dots,
t_p +v\int_0^1\frac{(x +s u)^p}{p!} ds\bigg).
  \end{aligned}
 \end{equation*}
 In particular 
\begin{equation}\label{esone} 
 \exp\Big(\xi X+\eta Y \Big)=\Big(\xi, \eta, \frac{\eta\xi}{2}, \frac{\eta\xi^2}{3!},
 \dots,\frac{\eta\xi^p}{(p+1)!}\Big).
\end{equation} 
Iterating the computation, we discover that the point $ \Gamma^{(p+1)}(w_1, \dots, w_{p+1})
\in\R^{ p+2 }$  takes the form 
\begin{equation*}
 \begin{bmatrix} 
u_1+ u_2 +\cdots +u_{p+1} 
\\
v_1 +v_2+ \cdots +v_{p+1} 
\\
v_1\displaystyle\int_0^1 s u_1 ds+ v_2\int_0^1(u_1+s u_2 )ds +\cdots 
+v_{p+1} \int_0^1(u_1+ u_2 +\cdots + s u_{p+1}) ds
\\
\displaystyle v_1\int_0^1 \frac{(su_1)^2}{2!} ds+ v_2\int_0^1\frac{(u_1+s u_2 )^2}{2!}ds +\cdots 
+v_{p+1} \int_0^1\frac{(u_1+ u_2+\cdots + s u_{p+1})^2}{2!} ds
\\ \vdots    
\\
 \displaystyle v_1\int_0^1 \frac{(su_1)^p}{p!} ds+ v_2\int_0^1\frac{(u_1+s u_2 )^p}{p!}ds +\cdots 
+v_{p+1} \int_0^1\frac{(u_1+ u_2+\cdots + s u_{p+1})^p}{p!} ds
 \end{bmatrix}
\end{equation*}
%
In order to calculate the matrix $M(\zeta)$, we write the first column in the  form 
$\frac{\p \Gamma^{(p+1)}}{\p u_1}(\xi, \dots ,\xi)=[1, *, \dots, *]^T$, where the terms $\ast$ are unimportant in the computation of the rank. All other variables $v_1, v_2,\dots, v_{p+1} $ appear linearly. Then it is easy to see that
 \begin{equation*}
  M(\xi,\eta) = 
  \begin{bmatrix}
1 &  0 &   0&\cdots & 0
\\ 
* & 1 & 1&\cdots & 1
\\
* & \xi\int_0^1 s ds & \xi\int_0^1(1+ s) ds &\cdots & \xi\int_0^1(p+s)ds
\\
\vdots & \vdots & \vdots &\ddots & \vdots 
\\ 
* & \xi^p\int_0^1 \frac{s^p}{p!} ds 
&  \xi^p\int_0^1 \frac{(1+s)^p}{p!}ds &\cdots & \xi^p \int_0^1
\frac{(p+s)^p} {p!}  ds
  \end{bmatrix}.
 \end{equation*}
In order to check the nonsingularity, we look at the submatrix obtained by  deleting the first row and column. Since $\xi\neq 0$, it suffices to check the nonsingularity 
 of the square  matrix of order $p+1$
 \begin{equation*}
\begin{aligned}
  \wh M : &= 
  \begin{bmatrix}
  1 & 1&\cdots & 1
\\
 2\int_0^1 s ds & 2\int_0^1(1+ s) ds &\cdots &  2\int_0^1(p+s)ds
\\
\vdots
\\ 
 (p+1) \int_0^1  s^p ds 
&  (p+1) \int_0^1  (1+s)^p ds & \cdots &(p+1) \int_0^1(p+s)^p ds
  \end{bmatrix}
  \\& =
  \begin{bmatrix}
1&1&1&\dots & 1
\\
1& 2^2-1^2 & 3^2 - 2^2&\dots &(p+1)^2 - p^2
\\ \vdots &\vdots &\vdots &\cdots&\vdots
\\ 1& 2^p-1^p & 3^p - 2^2&\dots &(p+1)^p- p^p
\end{bmatrix},
\end{aligned}
 \end{equation*}
 whose determinant, after some trivial column operations, is equal to a nonsingular Vandermonde determinant.  \end{proof}



\section{Inner cone property for horizontally convex sets }
  \subsection{Existence of inner cones to convex sets  }

  \begin{proof}[Proof of Theorem~\ref{teoremetto}]
 The proof of the inner cone property  \eqref{ggo}  is based on a 
 modification of  the arguments of \cite[Section~2.2, proof of Theorem~1.1]{Morbidelli18}.

 Let us consider $\ol x\in\p A$, assume that  for some $p\in\N$ and  $\xi\in\R^m\setminus
 \{0\}$,  the openness condition~\eqref{saba} holds. Assume also that we have  $\ol x\cdot \exp(p\xi\cdot X)\in \Int A$. This means that for some $\rho>0$ we have 
  $B(\ol x\cdot\Gamma^{(p)}( \xi,\dots, \xi),\rho)\subset \Int(A)$.  
By continuity, there is  $\e>0$ such that  if 
\begin{equation}\label{epsai} 
 \max\big\{d(\ol x,y ),|u_j-\xi|  :j=1,\dots, p\big\}<\e
\end{equation}
then
\[  y\cdot \Gamma^{ (p)}( u_1, u_2, \dots, u_{p-1}, u_p)\in B(\ol x\cdot\Gamma^{(p)}( \xi,\dots,\xi),\rho)\subset \Int(A).
\]

    We organize the proof in four steps.

    \step{Step 1.} 
    We claim that for all $x\in A$ with $d(x,\ol x)<\e$ we have 
   \begin{equation}
    \label{includo}
     x\cdot \Gamma^{(p)} (\lambda_1 w_1,\dots, \lambda_p w_p) \in B(\ol x\cdot 
    \Gamma^{(p)}( \xi,\dots, \xi),\rho)\subset \Int(A),
   \end{equation}  
for all   $\lambda_1,\dots, \la_p\in[0,p]$ such that $\sum_{j=1}^p\lambda_j=p$ and 
  $w_1,\dots, w_p$ such that $\max_j|w_j-\xi|<\e$. 
   
    Let us consider for $p\in\N$
    and 
    $C>0$ the set 
    \begin{equation*}
    \begin{aligned}
 K_C: & = \Big\{(x, \lambda_1, \lambda_2,\dots, \lambda_p, w_1, w_2, \dots, w_p)
     :
    |x |+\sum_{j=1}^p |w_j |\leq C , \lambda_j\geq 0 ,
    \sum_{j=1}^p\lambda_j=p
    \Big\}.         
    \end{aligned}
    \end{equation*}
For any $C>0$ the set $K_C$ is compact. Furthermore,  for any $C$, the function 
    \[\begin{aligned}
K_C \ni    (x,\lambda_1,\dots, \la_p, w_1, \dots, w_p) 
    &\mapsto
     x\cdot\exp(\la_1 w_1\cdot X)\cdot \cdots \cdot \exp(\la_p w_p\cdot X)
         \end{aligned}
 \]
depends  in a polynomial way  from its arguments. 
Therefore, given any
  $\ol x\in   \G$, $\xi \in\R^m$,
 $\la_1,\dots, \la_p\geq 0$ with $\sum_{j=1}^p\la_j=p$, there is $\sigma>0$ such that  
    \begin{equation}\label{quattrocchi} 
\Big| 
x \cdot   \Gamma^{(p)}  (\lambda_1 w_1 ,\la_2  w_2 , \dots, \la_p w_p )
- \ol x\cdot \Gamma^{(p)} (\lambda_1 \xi  , \lambda_2 \xi , \dots,\lambda_p \xi )
\Big|<\rho
    \end{equation} 
if $| x-\ol x|<\sigma$, $\lambda_j\geq 0$ for $j\in\{1,\dots, p\}$, with 
$\sum_j\lambda_j=p$ and $\max_j|w_j-\xi |<\sigma$.
Then the equality  \[\ol x\cdot \Gamma^{(p)} (\lambda_1 \xi  , \lambda_2 \xi , \dots,\lambda_p \xi )=\ol x\cdot \exp(p\xi \cdot X ) \]  and  a    choice of small $\e  $ in \eqref{epsai} gives the inclusion~\eqref{includo}.

\step{Step 2.} We claim that for all $x\in   A $ with $d(x,\ol x) < \e$ and for all $\lambda\in\left]0,1\right]$ we have 
\begin{equation*}
\Big\{  x\cdot\Gamma^{(p)}(\lambda u_1,\lambda u_2,\dots, \lambda u_p) :
\max_{  j\in\{1,\dots,p\} }|u_j-\xi|<\e\Big\}\subset A.
\end{equation*}
The proof is the same  presented in \cite{Morbidelli18} and works as follows. Let us look at any  point $x \in A$ such that $d(x ,\ol x)<\e $. Consider also the point  $x\cdot\Gamma^{(p)}
(pu_1,0,\dots,  0)$. This point belongs to $A$ and is aligned with $x\in A$. Then the horizontal segment connecting these two points, and in particular the point $ x\cdot\Gamma^{(p)}(\lambda u_1,0,\dots,  0) $ belongs to $A$. 

Next we repeat the  argument considering the pair of  points $ x \cdot\Gamma^{(p)}(\lambda u_1,0,\dots,  0)\in A$
and $ x \cdot\Gamma^{(p)}(\lambda u_1,(p-\lambda)u_2,0,\dots,  0)$, which belongs to~$A$ by Step 1. Since both these points belong to $A$ and are aligned, 
we deduce that the horizontal segment connecting them is contained in $A$. In particular $x\cdot \Gamma^{(p)}(\lambda u_1,\lambda u_2,0,\dots,  0)$. An iteration of the argument completes the proof of Step 2.

\step{Step 3.} Following \cite{Morbidelli18}, by   dilation and translation arguments in Carnot groups,  for a suitable $\delta_0>0$ we get the inclusion 
\begin{equation}\label{caracollo} 
\begin{aligned}
 \big\{x \cdot \Gamma^{(p)}(\lambda u_1,\dots, \lambda  u_p) & : \lambda\in\left]0,1\right] 
 \text{ and }\max_{1\leq j\leq p}|u_j-\xi |<\e\big\}
 \\&\supset \bigcup_{\lambda\in  \,  ] 0,1  ]}B\Big(x\cdot\exp(\lambda p\xi\cdot X), \delta_0\lambda\Big),
\end{aligned}
\end{equation}
which gives ultimately the proof of~\eqref{ggo}.  

  Let us check the desired inclusion following \cite{Morbidelli18}. 
 Let $x\in A$ with $d(x,\ol x)<\e$. 
 Then denoting by $L_x:\G\to \G$ the left translation,  $L_x y=x\cdot y $ for all $x,y\in\G$, starting from~\eqref{caracalla}, 
 we have
 \begin{equation*}
\begin{aligned}
 L_x\Big(\Big\{ & \Gamma^{(p)}(\la u_1, \dots, \la u_p):\la\in\left]0,1\right]\;\text{and }\max_{j\leq p}|u_j-\xi|<\e\Big\}\Big) 
 \\&= L_x\Big(\bigcup_{0<\la\leq 1 }\delta_\la
 \Big\{   \Gamma^{(p)}(  u_1, \dots,   u_p): \max_{j\leq p}|u_j-\xi|<\e\Big\}\Big)
 \\&\supseteq \bigcup_{0<\la\leq 1} L_x\big(\delta_\la B_\Eucl(\Gamma^{(p)}(\xi,\dots,\xi),\s_\e)\big)
 \supseteq \bigcup_{0<\la\leq 1} L_x\big(\delta_\la B (\Gamma^{(p)}(\xi,\dots,\xi),c_0 \s_\e)\big).
\end{aligned}
 \end{equation*}
The penultimate inclusion follows from  the local openness assumption. The last is a consequence of the standard local estimates for  distances defined by vector fields. \footnote{ If $d$ is the subRiemannian distance defined by a given family of $C^1$ vector fields $X_1, \dots, X_m$ and $B$ denotes the corresponding ball, then for any compact set $K\subset\R^n$ there is $c_0>0$ such that $B_\Eucl(y, r)\supseteq B(y, c_0 r) $ for all $r\leq c_0$ and $y\in K$.   }
 Thus,   since $\delta_\la \Gamma^{(p)}(\xi, \dots, \xi)=\exp(\la p\xi \cdot X)$,  
 we have finished the proof of~\eqref{caracollo}.

\step{Step 4.} Until now we proved the inner cone inclusion for vertices  $x\in A$. 
By an approximation argument, we can approximate any point $x\in\p A$
with $d(x,\ol x)<\e$ with a family $x_n\in A$ for all $n\in\N$ such that $x_n\to x$ as 
$n\to \infty$. Since the aperture of the cones are stable  as $n\in\N$, we get inclusion~\eqref{ggo} for $x\in\p A$. Note that we are not assuming that $A$ is closed.  
 \end{proof}

  \subsection{Examples     of failure of the cone property     -- the filiform case}\label{ggiorgio} 
  In this section we  consider the pair of vector fields   
  \begin{equation*}
   X=\p_x \quad \text{and }\quad Y=\p_y + x\p_{t_1}+\cdots +\frac{x^p}{p!}\p_{t_p}
  \end{equation*}
described in Section~\ref{zibibbo}. We look at the   direction $Y$ and we  show an example where   Theorem~\ref{teoremetto} fails at that direction, for some convex sets.   This gives also an indirect     proof of the fact that the    for any $p\in\N$, the $p$-th multiexponential cannot be locally open  
at $((0,1), \dots, (0,1))\in (\R^2)^p$.   

%

\begin{example} \label{quattrouno} 
Let $p\geq 2$. Assume first that $p$ is even and let us look  at the set 
\begin{equation}\label{contraerea} 
   E=\{(x,y, t_1, \dots, t_p)\in\R^{p+2}: F(x,y,t):= t_p +y^{p+2}     
   \mathbb{1}_{\left[0,+\infty\right[}(y)\geq 0\}.                                                                 
   \end{equation} 
It is easy to check that  $XF=0$ identically, and  
\begin{equation}\label{gs} 
YF(x,y,t_1, t_2, 
\dots, t_p)=\frac{x^p}{p!}+ (p+2)y^{p+1}\mathbb{1}_{\left[0,+\infty\right[}(y)
\geq 0,
\end{equation}  
because  $p$ is even.   It follows that  both $E$ and $E^c$ are  horizontally convex. 
(The set has also constant horizontal normal, see~\cite{FSSC03,BaroneAdesiSerraCassanoVittone,BellettiniLedonne12} and the 
 references therein for the related definition).

If we consider the point $P=0\in \p E$,  the point $Q:=\exp(Y)=(0,1,0,\dots, 0)\in\Int(E)$ and  the curve $\gamma(s)=\exp (sY) = (0,s,0,\dots, 0)$, it turns out that $\gamma(s)\in\p E $ for all $s\leq 0$ and $\gamma(s)\in\Int(E)$ for all $s>0$. However for any $\e>0$ and $s_0>0$ the inclusion 
  \[
\bigcup_{0<s<s_0}   B((0,s,0,0,\dots, 0), \e s)\subset E
  \]
fails. Indeed, by the translation law~\eqref{groppo} and the standard ball-box  
theorem, $B((0,s,0,\dots),\e s)$ contains all points of the form 
$P_s:=(0,s,0,0,\dots, ,-c(s\e)^{p+1})$ for some universal $c>0$. Instead, the point $P_s$  can not belong to the set $E$ for $s$ belonging to any   nontrivial interval with left extremum $0\in\R$.

Even more strikingly, if we choose $P=\exp(-Y)=(0, -1,0, \dots)\in \p E$ and $Q=P\exp(2Y)= (0,1,0,\dots)\in \Int(E)$, we see that even the much weaker qualitative  property $\{\exp(sY):s\in \left]-1,1\right[\}\subset \Int(E)$ fails.
\end{example}
  
  If $p\geq 3$ is odd then the set in \eqref{contraerea}   
  is not  horizontally convex.    To check this claim it suffices to take 
  $  \gamma(s)=\exp(s(-X+Y))=(-s, s,-\frac{s^2}{2}, \frac{s^3}{3!},-\frac{s^4}{4!}, \dots, -\frac{s^{p+1}}{(p+1)!})$,
  by \eqref{esone}.  It is easy to see that the    path $\gamma$ satisfies $\gamma(0)\in E$, $\gamma(1/(p+1)!)\in E$ and 
  $\gamma\Big(\bigl]0,\frac{1}{(p+1)!}\bigr[\Big)\subset E^c$.
 However  the discussion concerning the set defined in~\eqref{contraerea} 
can be modified by taking     
\[
E=\{ t_{p-1}+ y^{p+1}\mathbb{1}_{\left[0,+\infty\right[}(y)\geq 0\}
\]
and arguing as above.

\begin{remark} \label{quatt} 
 In the Engel group $\mathbb{E}=\R^4$ with  vector fields $X_1=\p_1$ and $X_2 = \p_2+ x_1\p_3 +\frac{x_1^2}{2}\p_4$,    and group law
 \begin{equation}\label{grl} 
  x\cdot y=\Big(x_1+y_1, x_2+y_2, x_3 + y_3 + x_1 y_2, x_4+y_4 +\frac{x_1^2}{2}y_2+x_1y_3\Big),
 \end{equation}  
 the analogous example is given by $x_4>\psi(x_2)$ with $\psi_2'\leq 0$. See \cite{BellettiniLedonne12} where many examples of constant horizontal normal sets are exhibited. In such case a  counterexample to the cone property is given by $E=\{x_4 > - x_2^4 \mathbb{1}_{\left[0,+\infty\right[}(x_2)\}$ where the inner cone property does not hold.
\end{remark}

\begin{remark}
Let us observe that the failure of the cone property at some direction $V\in V_1$, as in the examples discussed here, does not imply that the curve $\gamma(s)=\exp(sV)$ is $C^1$-rigid in the sense of Bryant and Hsu~\cite{BryantHsu}. In this regard, given the Engel group~$\mathbb{E}$ of Remark~\ref{quatt},  let us consider the direct product $\G=\R\times\mathbb{E}\ni(x_0, x)$ with group law 
$ (x_0, x)\cdot (y_0, y):= (x_0+y_0,x\cdot y )$ and with horizontal vector fields
$ X_0=\p_{0}, 
$ 
$X_1=\p_1$ and $X_2 = \p_2+ x_1\p_3 +\frac{x_1^2}{2}\p_4$.
It is easy to see that the set 
\begin{equation*}
 E=\{(x_0, x)\in \R\times \mathbb{E}: x_4> -x_2^4\mathbb{1}_{\left[0,+\infty\right[}(x_2) \}
\end{equation*}
does not satisfy the cone property in the direction $X_2$,  (see Example~\ref{quattrouno}), but the curve $\gamma(s)=\exp(sX_2)=(0, (0,s,0,0))$ with $s\in[0,1]$ is not $C^1$-rigid, because it can be perturbed smoothly with horizontal curves of the form
\begin{equation*}
\wt\gamma(s)=(\eta(s), (0,s,0,0)) ,
\end{equation*}
with $\eta $ is an arbitrary smooth, compactly supported  function in $\left]0,1\right[$.

\end{remark}
 
  \subsection{  Examples     of failure of the cone property    -- the free group
  of step three and  rank  two  }
  \label{gggiorgio} Here we show in the model of free three-step Carnot group with two generators  an example where  Theorem~\ref{teoremetto} fails. 
 The following class of examples are  minor modifications of the examples of Section~\ref{ggiorgio}.    
  
  Consider in $\R^5$ with variables $(x_1, x_2, x_3, x_4, x_5)$ the vector fields 
  \begin{equation}
   X_1=\p_1-\frac{x_2}{2}\p_3 -\frac {x_1^2+x_2^2}{2}\; \p_5
   \quad\text{and}\quad X_2=\p_2+\frac{x_1}{2}\p_3 +\frac{x_1^2+x_2^2}{2}\;\p_4
  \end{equation} 
  which together with their commutators
  \begin{equation*}
  X_3:= [X_1, X_2]=\p_3+x_1\p_4+x_2\p_5,\quad
   X_4:=[X_1, X_3]=\p_4\quad\text{and}\quad X_5:=[X_2, X_3]=\p_5
  \end{equation*}
generate the free Lie algebra of step three with two generators and are left invariant with respect to the law
\begin{equation}
\begin{aligned}
 x\cdot y=\bigg(  & x_1+y_1, x_2+ y_2, x_3 + y_3 +\frac 12(x_1y_2 -x_2 y_1), 
 \\& x_4+ y_4 +\frac {y_2}{2}  (x_1^2+x_2^2+x_1y_1+x_2 y_2) + x_1 y_3,
\\& x_5+ y_5 -\frac{y_1}{2}(x_1^2+x_2^2 + x_1 y_1 + x_2 y_2 ) +  x_2 y_3\bigg).
\end{aligned}
\end{equation} 
This model has been studied by Sachkov~\cite{Sachkov03}  ~\cite{ArdentovLeDonneSachkov19}.

A standard computation gives for all $(\xi_1, \xi_2)\in\R^2$ 
\begin{equation*}
 \exp\Big(s(\xi_1 X_1+\xi_2 X_2)\Big)=\Big(
 \xi_1 s, \xi_2 s, 0, \frac{\xi_2}{6}(\xi_1^2+\xi_2^2)s^3, -\frac{\xi_1}{6}
 (\xi_1^2+\xi_2^2)s^3
 \Big).
\end{equation*}
It is well known that in this model all integral curves $\gamma(s) = \exp(s(\xi_1 X_1+\xi_2 X_2))$ 
are normal and abnormal minimizers.  
Therefore the construction of the multiexponential map does not provide the inner 
cone property.  In the following discussion we present some examples of sets where 
inclusion \eqref{ggo} fails. 

\begin{lemma}
 Let $\psi:\R\to\R$ be a nonincreasing  regular function.
 Then, for any fixed  unit vector   $\xi:=(\xi_1, \xi_2)\in\R^2$, 
 the set 
 \begin{equation}\label{esepp} 
  E:= \Big\{x= (x_1, \dots, x_5):  F(x_1, \dots, x_5):= \xi_2 x_4 - \xi_1 x_5  -\frac{\scalar{\xi}{x}^3}{6}-\psi(\scalar{\xi}{x})  
   >0 \Big\}
 \end{equation}   together with its complementary $E^c$ 
 is horizontally convex.  
\end{lemma}
  Note that it turns out from the proof that the set has constant horizontal normal too.   In  the statement  and below we denoted $\scalar{\xi}{x}=\xi_1 x_1+\xi_2 x_2$.
\begin{proof}
 Let $F(x)=\quad\xi_2 x_4 - \xi_1 x_5  -\frac{\scalar{\xi}{x}^3}{6}-\psi(\scalar{\xi}{x})$. A trivial computation shows that  
 \begin{equation*}
\begin{aligned}
  (-\xi_2 X_1 + \xi_1 X_2 )F &=0\quad\text{ and }\quad 
  \\( \xi_1 X_1 +\xi_2 X_2) F & = \frac{\abs{\xi}^2}{2}\Big( x_1^2 + x_2^2  -\scalar{x}{\xi}^2\Big )-\abs{\xi}^2\psi'\big(\scalar{x}{\xi}\big)\geq 0,
\end{aligned}
 \end{equation*}
because $\psi$ is nonincreasing and $|\xi|=1$. Therefore the set has constant horizontal normal and   in particular both~$E$ and~$E^c$ are horizontally convex.
\end{proof}

\begin{example}
 Let $\xi\in\R^2$ be a unit vector and let us consider the set $E$ defined in \eqref{esepp}. Let us choose the function $\psi(t)= -t^4 \mathbb{1}_{\left[ 0,+\infty\right[}(t)$, so that the set $E$ becomes 
 \begin{equation*}
  E:=\Big\{x: \xi_2 x_4 -\xi_1 x_5-\frac{\langle\xi,x\rangle^3 }{6} + 
  \langle\xi,x\rangle^4\mathbb{1}_{\left[0, +\infty\right[}(\langle \xi, x\rangle)>0 
  \Big\}.
 \end{equation*}
  Here the origin $0$ belongs to $ \p E$, while 
 \[
  \exp(s\xi\cdot X)=\Big(s\xi_1, s\xi_2, 0, \frac{\xi_2}{6}s^3,-\frac{\xi_1}{6}s^3  
  \Big)\in\Int (E) \quad \text{ for all $s>0$.}
 \]
Assume that there exist positive numbers
$\e $ and $ s_0 $ such that 
\[
C_{\e, s_0} := \bigcup_{0<s<s_0}B\Big ( \exp(s\xi\cdot X) ,\e s\Big)\subset E
\]
for all $s>0$. We claim that this gives a contradiction. 
The cone $C_{\e, s_0}$  must contain all points of the form $\exp(s(\xi_1 X_1+\xi_2 X_2)) 
\cdot (\e s u_1,\e s u_2,\e^2 s^2 u_3 , \e^3 s ^3u_4,    \e^3 s^3 u_5)$, where $|u|\leq c$ and $c>0$  is an absolute constant. In particular,  
%
\begin{equation*}
  \begin{aligned}
C_{\e, s_0}  \supseteq   &\Big(s\xi_1, s\xi_2, 0, \frac{\xi_2}{6}s^3,-\frac{\xi_1}{6}s^3  \Big)
 \cdot (0,0, 0, - c\xi_2  \e^3 s ^3,  c\xi_1  \e^3 s ^3)
 \\&    
 =\Big(s\xi_1, s\xi_2, 0, \xi_2 s^3\Big(\frac{1}{6}-c\e^3\Big), -\xi_1 s^3\Big(
 \frac 16-c\e^3\Big) \Big)=:\gamma(s),\end{aligned}
\end{equation*}
where we  recall again that $\xi_1^2+\xi_2^2=1$.  
An elementary computation shows that   for $s>0$ we have    $\gamma(s)\in E$ if and only if 
$-c\e^3 s^3 + s^4>0$ and this inequality fails for $s\in\left]0, c\e^3\right[$.
In other words  for any $\e>0$ fixed, the point $\gamma(s)$  does not belong to the set~$E$ defined in~\eqref{esepp}   for positive $s$  close to $0$.
\end{example}

\section{Differentiability of the distance}
\label{dudu} 
\subsection{Proof of Theorem \texorpdfstring{\ref{oppiaceo}}{14} }

%
%
 \begin{proof}[Proof of Theorem~\ref{oppiaceo}]
Let $(\R^n,\cdot)$ be a Carnot group of step $s$. Write $x=(x^1, x^2, \dots, x^s)
\in\R^n = \R^{m_1}\times\R^{m_2}\times\cdots\times\R^{m_s}$. Denote for brevity $m=m_1$.  Assume that $d\Gamma^{(p)}(\xi,\xi,\dots,\xi):(\R^m)^p\to\R^n$  is onto for some given $\xi\in\R^m$.   This is equivalent to require that the map 
$d\Gamma^{(p)}(\lambda \xi,\lambda\xi,\dots,\lambda \xi):(\R^m)^p\to\R^n$
 is onto for any $\lambda>0$.    Let $w=:p\xi$. We want to show that 
\begin{equation}\label{dof} 
 d\Big(\exp(w\cdot X)\cdot x\Big)= d\Big(\exp(w\cdot X) \Big)+\Big\langle\frac{w}{|w|},x^1 \Big\rangle_{\R^m}+o(d(x)), 
\end{equation} 
as $x\to 0\in\R^n$.  We adopt here and hereafter the standard notation $d(x):=d(0,x)$. In~\cite[Lemma~3.2]
{LeDonnePinamontiSpeight17} it has been proved  
that the lower estimate $\geq$ in~\eqref{dof} holds in any Carnot group  of arbitrary step and for all choice of $\xi\in\R^m\setminus\{0\}$. Therefore, we discuss here the upper estimate only. If~\eqref{dof} holds, then this  means that  the distance from the origin is Pansu differentiable at $\exp(w\cdot X)$ and its differential  is the map $T:\G\to \R$ defined by $T(x^1,\dots, ,x^s)=\big\langle\frac{w}{|w|}, x^1 \big\rangle_{\R^m}$. This explicit formula shows that  the differential is the same at any point $\exp(\lambda w\cdot X)$ for any~$\lambda>0$,   as it happens in the Euclidean case.

Let us discuss  the upper estimate in \eqref{dof}.  Let $\frac{w}{p}=\xi $.   Look at the map 
  \begin{equation}
\begin{aligned} (\R^m)^p \ni &  (\a_1,\a_2,\dots,\a_p)
\longmapsto F(\a_1,\a_2,\dots,\a_p )
\\  
&:=
\exp\big((\xi+\a_1)\cdot X\big)\cdots \exp\big((\xi+\a_p)\cdot X\big)\in\R^n.
\end{aligned}
  \end{equation} 
 Note that $F(0)=\exp(w\cdot X)$.    Since $dF(0,\dots, 0)$ is onto, there is a $n$-dimensional subspace $V\subset(\R^m)^p$ such that  $ dF(0) \Big|_V:V\to \R^n $ is invertible. By the inverse function theorem there is a neighborhood $U$ of the origin in $\R^n$ such that   
for all $x\in U$ the system of equations 
\begin{equation}\label{sisto} 
\begin{aligned}
&\exp\big((\xi+\a_1)\cdot X\big)\cdots \exp\big((\xi+\a_p)\cdot X\big) 
= \exp(w\cdot X)\cdot x
\end{aligned}
\end{equation}
has a unique solution $( \ol\a_1,\dots   ,\ol\a_p)\in V$  
which satisfies,   for suitable constants $C$ and $\wh C>0$,  
\begin{equation}\label{upload} 
\begin{aligned}
 \abs{(\ol\a_1, \dots,\ol \a_{p })}_{\Eucl}
 & \leq C \;\bigl |\exp(w\cdot X)\cdot x -\exp(w\cdot X)   \bigr|_{\Eucl}
 \\&  \leq \; \wh C d  \Big( \exp(w\cdot X)\cdot x, \exp(w\cdot X) \Big)
 =\wh C d(x),
\end{aligned}
\end{equation}
by standard subRiemannian facts.   In the formula above, we denote by $|\cdot|_{\Eucl}$ the Euclidean norm.   

By definition of distance we have
\begin{equation*}
\begin{aligned}
 d\Big( &\exp (w\cdot X)\cdot x \Big)
\leq 
\sum_{j=1}^p|\xi+\ol \a_j|=p|\xi|+\bigg\langle \frac{\xi}{|\xi|}, \sum_{j=1}^p \ol\a_j \bigg \rangle+O(|\ol\a|^2),
\end{aligned}
\end{equation*}
by the Taylor formula, as $x\to  0 $.
  Formula~\eqref{upload} tells that    $O(|\ol\a|^2)=O(d(x    )^2)$.   

 Recall also that $d(\exp(w\cdot X))=\abs{w}= p  \abs{\xi }$.  
 A look to the first $m$ equations of the system~\eqref{sisto} gives also the equality $\sum_{j=1}^p\ol\a_j =x^1\in\R^m$. Therefore, we have obtained the inequality
\begin{equation*}
d\big(\exp(w\cdot X)\cdot x\big)\leq d(\exp(w\cdot X))+\bigg\langle \frac{w}{|w|}, x^1\bigg\rangle_{\R^m}+O(d(x)^2), 
\end{equation*}
which concludes the proof. \end{proof}

 \subsection{The step-two case}\label{duofold} 
Here we prove Theorem \ref{pasodos}, stating that in Carnot groups of step two the subRiemannian distance is differentiable at any point $\exp(W)$ for any $W\in V_1$. 
 As we already observed, here we are   able to get the differentiability  also when $s\mapsto \exp(sW)$ is a singular subRiemannian length-minimizer. The theorem was first proved  in \cite{LeDonnePinamontiSpeight17}, but our proof relies on a different argument.

Let us consider a Carnot group of step two. Namely, equip $\R_z^m\times\R_t^\ell$  
with the group law~\eqref{steppa}
 \begin{equation*}
  (z,t)\cdot (\z,\tau)=(z+\z, t+\tau + Q(z,\z))
  \in\R^m\times\R^\ell.
 \end{equation*} 
 See Section~\ref{sezione2} for further details.  In the sequel we will use several times the fact that the bilinear function $Q$ satisfies the alternating property $Q(z,z)=0$ for all $z\in\R^m$.

An easy computation based on the skew-symmetry of $Q$ gives $(w,0)=\exp (w\cdot X)$, where $w\in\R^m$ and $w\cdot X:=\sum_{j=1}^m w_jX_j$.  For any $w\in\R^m\setminus\{0\}$    we want to get the estimate  
\begin{equation}\label{addis} 
 d((w,0)\cdot(z,t))\leq    d(w,0)  +\Big\langle \frac{w}{|w|}, z\Big\rangle + o(d(z,t))\quad \text{ as $(z,t)\to (0,0)$}.
\end{equation} 
%
 Recall again that the opposite inequality holds in general Carnot groups, see~\cite[Lemma~3.2]{LeDonnePinamontiSpeight17} and~\cite[Lemma~2.11]{PinamontiSpeight18}.
 
 In order to prove \eqref{addis}, we analyze the multiexponential map
 \begin{equation*}
   \Gamma^{(p)}(u_1,\dots, u_p):=\exp(u_1\cdot X)\cdots \exp(u_p\cdot X) =\Bigl(\sum_{j\le p}u_j,
 \sum_{1\le j<k\le p}Q(u_j,u_k)
 \Bigr),
 \end{equation*} 
where $p\in\N$ will be chosen later on,  and the vectors $u_1, \dots, u_p$ belong to  $\R^m$.

Our purpose is to analyze the system $\Gamma^{(p)}(\xi + u_1, \dots, \xi+u_p)=(w,0)\cdot (z,t)$,  where $ \xi:=\frac{w}{p}$, in order to get the upper estimate~\eqref{addis}.
Using the group law  we get the set of equations
 \begin{equation}\label{twostep} 
  \Bigl(\sum_{j\le p}(\xi+u_j),
 \sum_{1\le j<k\le p}Q(\xi+u_j,\xi+u_k)
 \Bigr) =(w+z,t + Q(w,z) ).
 \end{equation} 
 After a short manipulation, we get
\begin{equation}\label{jowo} 
 \left\{
 \begin{aligned}
 & \sum_{j=1}^p u_j = z 
 \\
 &Q\Big(\sum_{j=1}^p (p-2j+1)u_j, \xi\Big) +\sum_{1\leq j<k\leq p}Q(u_j, u_k)
 = t+Q(p\xi, z).
        \end{aligned}
 \right.
\end{equation} 
By definition of subRiemannian distance, a solution $u_1, \dots, u_p$ of~\eqref{jowo} provides immediately the estimate $d((w,0)\cdot(z,t)) \leq \sum_j|\xi+u_j|$. 
Besides this trivial remark, 
the key point in the proof of \eqref{addis} is the following proposition.
\begin{proposition}\label{appro} 
There  are $p\in\N$ and $C>0$ such that for all   $\xi\in\R^m$  and for each 
   $(z,t)\in\R^m\times\R^\ell$, the system~\eqref{jowo} has a solution  $(u_1, \dots, u_p)$ satisfying the inequality 
\begin{equation}\label{jowone} 
 \sum_{j=1}^p|u_j|\leq C(|z|+|t|^{1/2}).
\end{equation} 
\end{proposition}
By standard facts, $|z|+|t|^{1/2}$ is equivalent to $d(z,t)$.
 In \cite[Theorem~2.1]{Morbidelli18} the second author solved a system similar to~\eqref{jowo},  but without the term $Q(p\xi, z)$.     Unfortunately, the estimates of the mentioned paper are not sufficient to discuss the present case. Furthermore,  here we find a method of solution which is much simpler than the one in \cite{Morbidelli18}.

Before proving Proposition~\ref{appro} we show how such result gives the required estimate~\eqref{addis}.  
\begin{proof}
 [Proof of Theorem \ref{pasodos}]
Let us fix $(w,0)\in\R^m\times\R^\ell$. Let $(z,t)$ and take a solution of~\eqref{jowo} satisfying~\eqref{jowone}. Using the definition of control distance and the Euclidean Taylor formula we discover that 
 \begin{equation}
\begin{aligned}
  d((w,0)\cdot(z,t)) &\leq \sum_{j=1}^p |\xi+u_j|
  =\sum_{j=1}^p \Big( |\xi| +\Big\langle u_j, \frac{\xi}{|\xi|}\Big\rangle +O(|u_j|^2)\Big)
  \\&=|p\xi| +
  \Big\langle z, \frac{\xi}{|\xi|} \Big\rangle +O(|z|^2+|t|) ,
\end{aligned}
 \end{equation} 
 which is the required inequality~\eqref{addis}. 
 \end{proof}
 
 \begin{proof}[Proof of Proposition \ref{appro}]
 It suffices to show that 
there is $C>0$ such that for all $(z,t)\in\R^m\times\R^\ell$, the system 
  \begin{equation}\label{quattordici} 
 \left\{
 \begin{aligned}
 & \sum_{j=1}^p u_j = z 
 \\&
 \sum_{j=1}^p (p-2j+1)u_j=-p z 
  \\
 &\sum_{1\leq j<k\leq p}Q(u_j, u_k)
 = t 
        \end{aligned}
 \right.
 \end{equation} 
 has a solution that satisfies estimate~\eqref{jowone}.  
 Note that the system~\eqref{quattordici} does not contain~$\xi$. Therefore our final estimates will be independent of $\xi\in\R^m$.

 Observe now that the second equation of \eqref{quattordici}, combined with the first, can be written in the form  
 \begin{equation}
  \label{quindici}
  \sum_{j=1}^p ju_j = \frac{1+2p}{2}z
 \end{equation}

    Let us make the linear change of variable 
\begin{equation*}
 v_1=u_1,\quad v_2 = u_1 + u_2, \quad \dots,v_k=\sum_{j=1}^k u_j=v_{k-1}+  u_k, \quad\text{up to $k=p.$}
\end{equation*}
Therefore, we have 
\begin{equation*}
 \sum_{j=1}^{p-1}v_j = \sum_{k=1}^{p-1}(p-k)u_k=p\sum_{k=1}^p u_k -\sum_{k=1}^p ku_k=pz
 -\sum_{k=1}^p ku_k.
\end{equation*}
Comparing with \eqref{quindici}, we discover that the   
 first two equations of the system~\eqref{quattordici} become 
\begin{equation}\label{ells} 
 v_p = z \quad
\text{ and }\quad 
 \sum_{j=1}^{p-1}v_{j}=-z/2. 
\end{equation}
Since we would have no advantage in solving the problem with small $p$, we will feel  free to use large values of $p$ in the argument below.    
The quadratic part takes the form
\begin{equation}\label{soqquadro} 
\begin{aligned}
t&=  \sum_{1\leq j<k\leq p}Q(u_j, u_k)=\sum_{k=1}^{ p-1}Q(v_k, v_{k+1})
\\&
=\sum_{k\leq p-3}Q(v_k, v_{k+1})+ Q\big(v_{p-2}- v_p, v_{p-1}\big).
\end{aligned}
\end{equation} 
Let us choose $v_{p-1}=0$, so that the last term in~\eqref{soqquadro} vanishes. Fix also $v_{p-3}=0$. Then we have  fixed the set of conditions 
\begin{equation}\label{jocondor} 
 v_p=z,\quad v_{p-1}=0,\qquad v_{p-2}=-\frac{z}{2}-\sum_{j\leq p-4}v_j,\qquad v_{p-3}=0.
\end{equation} 
 Under all these choices, the first two equations of~\eqref{quattordici} are satisfied, while the quadratic part takes the easy form
 \begin{equation*}
  \sum_{j\leq p-5}Q(v_j, v_{j+1})= t,
 \end{equation*}
  where the variables $v_1, v_2, \dots, v_{p-4}$ are completely free.    Finally, taking $h\in\N$ and $p-5=1+3h$ and chooosing $v_3=v_6=v_{9}=\cdots
= v_{3h}=0$ for all $h\in\{ 1,2,\dots \}$, the system becomes
\begin{equation*}
 Q(v_1, v_2)+ Q(v_4, v_5)+ Q(v_7,v_{8})+\cdots 
 + Q(v_{1+3h}, v_{2+3h})=t,
\end{equation*}
which takes a pairwise decoupled form. Then it suffices to apply the H\"ormander condition, as in \cite[Lemma~2.3]{Morbidelli18}
 to see that   if $h\in\N$  is sufficiently large (depending on the algebraic strucure of 
 the group only) then     there is a solution 
 satisfying the required estimates  $|v_j|\leq C|t|^{1/2}$ for all $j\leq 
 2+3h=p-4$. The final terms $v_j$ with $j=p-3,p-2,p-1$ and $p$ can be estimated 
 by~\eqref{jocondor} with $C(|z|+|t|^{1/2})$.
%
%
%
%
%
 \end{proof}


  
\begin{remark} 
 In  \cite {PinamontiSpeight18}, Pinamonti and Speight introduce the notion
  of deformable direction in a Carnot group of step $s\geq 1$.
 We observe informally that from our results one can get the following two facts.
 \begin{itemize}[nosep]
  \item In any Carnot group,    if there are $p\in\N$ and $w\in\R^m\setminus\{0\}$ such that $\Gamma^{(p)}$ is a submersion at $(w,\dots, w)\in(\R^m)^p$, then   the direction $w\cdot X\in V_1$   is deformable.
  \item   If we restrict to   Carnot groups of step two, the discussion of 
  Section~\ref{duofold} proves   that  any horizontal direction is deformable. 
 \end{itemize}
Therefore, our results can be used to give another  proof of the deformability results in \cite{LeDonnePinamontiSpeight17,PinamontiSpeight18}. 
\end{remark}


%

\section*{Acknowledgements}  
 We thank the referee, whose remarks  contributed to some improvements of the paper. 
 
\medskip The authors are members of the {\it Gruppo Nazionale per
l'Analisi Matematica, la Probabilit\`a e le loro Applicazioni} (GNAMPA)
of the {\it Istituto Nazionale di Alta Matematica} (INdAM)

\footnotesize

 \phantomsection
\addcontentsline{toc}{section}{References}

\footnotesize

\def\cprime{$'$} \def\cprime{$'$}
\providecommand{\bysame}{\leavevmode\hbox to3em{\hrulefill}\thinspace}
\providecommand{\MR}{\relax\ifhmode\unskip\space\fi MR }
\providecommand{\MRhref}[2]{%
  \href{http://www.ams.org/mathscinet-getitem?mr=#1}{#2}
}
\providecommand{\href}[2]{#2}

\end{document}